\documentclass[11pt, reqno]{amsart}

\usepackage{geometry}
\usepackage{color}
\usepackage{amssymb}
\usepackage{amsmath}
\usepackage{arydshln}
\usepackage{bbm}
\usepackage{mathrsfs}
\usepackage{enumerate}

\usepackage[noadjust]{cite}

\usepackage{graphicx}
\usepackage{subfig}

\usepackage{hyperref}

\numberwithin{equation}{section}

\newcommand{\R}{\mathbb{R}}

\newcommand{\fP}{\mathbb{P}}
\newcommand{\E}{\mathbb{E}}

\newcommand{\cB}{\mathcal{B}}
\newcommand{\cD}{\mathcal{D}}
\newcommand{\cE}{\mathcal{E}}
\newcommand{\cF}{\mathcal{F}}
\newcommand{\cH}{\mathcal{H}}

\newcommand{\cL}{\mathcal{L}}
\newcommand{\cP}{\mathcal{P}}

\DeclareMathOperator{\Id}{Id}

\newtheorem{Theorem}{Theorem}[section]
\newtheorem{Proposition}[Theorem]{Proposition}

\newtheorem{Lemma}[Theorem]{Lemma}
\newtheorem{Remark}[Theorem]{Remark}

\newtheorem{Assumption}{Assumption}

\begin{document}

\title{Stability properties of some port-Hamiltonian SPDEs}


\author{Peter Kuchling}
\address[Peter Kuchling]{Faculty of Engineering and Mathematics\\ Bielefeld University of Applied Sciences and Arts, Germany}
\email[Peter Kuchling]{peter.kuchling@hsbi.de}

\author{Barbara R\"udiger}
\address[Barbara R\"udiger]{School of Mathematics and Natural Sciences\\ University of Wuppertal, Germany}
\email[Barbara R\"udiger]{ruediger@uni-wuppertal.de}

\author{Baris Ugurcan}
\address[Baris Ugurcan]{School of Mathematics and Natural Sciences\\ University of Wuppertal, Germany}
\email[Baris Ugurcan]{ugurcan@uni-wuppertal.de}

\date{\today. The final reviewed version of this manuscript was published in Stochastics under \cite{stochastics_publication}.}

\begin{abstract}
 We examine the existence and uniqueness of invariant measures of a class of stochastic partial differential equations with Gaussian and Poissonian noise and its exponential convergence. This class especially includes a case of stochastic port-Hamiltonian equations.
\end{abstract}

\maketitle

\section{Introduction}

Since L. E. Boltzmann we know that the thermodynamics of a gas can be derived from the microscopic dynamics of its interacting particles, each of which following the laws of classical mechanics. The impact with other particles acts as a port, which, due to the high number, is stochastic. The ergodic behavior of interacting particle systems with invariant measures has been studied in statistical mechanics \cite{FR_GLPbook}, \cite{FR_LP63}. By performing a law of large numbers (see e.g. \cite{FR_GLPbook}), the macroscopic equation obtained is the Fokker-Planck equation of a stochastic process. Further scaling limits of stochastic interacting particles take fluctuations around the deterministic macroscopic behavior into account, and are described by SPDEs \cite{FR_GLPbook,FR_DOPT94,FR_DOPT96}, having a form as \eqref{FR_PHSPDE} below. Critical fluctuations are also described by nonlinear SPDEs.

The well-posedness and stability properties of solutions of port-Hamiltonian SPDEs is subject of the research of this article. Not much is known about infinite-dimensional stochastic port-Hamiltonian-systems, see, e.g., \cite{lw_2},   \cite{FR_bookDMSB}.

To introduce stochastic port-Hamiltonian systems, let us review the finite-dimensional deterministic situation. It is well-known that a classical mechanical system can be described by the Hamiltonian equations
\begin{align*}
    \dot{q}&=\frac{d\cH}{dp}(q,p)
    \\
    \dot{p}&=-\frac{d\cH}{dq}(q,p)+F
\end{align*}
Where $q$ and $p$ denote position and momentum of the particles, respectively. $\cH$ is the Hamiltonian, which usually represents the total energy of the system, and $F$ is an outside force. For a conservative system, we have $F\equiv 0$. We note that the system may be written as
\begin{displaymath}
    \dot{x}=J\nabla_x \cH(x)+gF
\end{displaymath}
where $x=(q,p)^T$, $\nabla_x=(\nabla_q,\nabla_p)^T$, $g=(0,\Id)^T$ and
\begin{equation}\label{eq:def_J_id}
    J=\begin{pmatrix}
    0&\Id\\
    -\Id&0
    \end{pmatrix}.
\end{equation}
For the introduction of port-Hamiltonian systems, we remark two aspects:
\begin{enumerate}
    \item The matrix $J$ is skew-symmetric.
    \item Often times, the outside force only acts on the momentum coordinates, i.e. the force is transformed to act on a part of the system.
\end{enumerate}
These aspects make up the core of the theory of port-Hamiltonian systems.

One port that is often included when modelling physical systems is the so-called resistive port. Depending on the setting, it is used to describe e.g. friction or electric resistance. Due to its linear nature, under reasonable assumptions, it can be incorporated into the above system as follows:
\begin{displaymath}
    \dot{x}=A\nabla_x \cH(x)+gF
\end{displaymath}
where $x=(q,p)^T$, $\nabla_x=(\nabla_q,\nabla_p)^T$, $g=(0,\Id)^T$ and
\begin{equation}\label{eq:operatorA_phs}
    A=J-R
\end{equation}
for some positive definite matrix $R$ and a skew-symmetric matrix $J$, which generalizes \eqref{eq:def_J_id}. The definiteness of $R$ is due to the model, since it represents resistance.

The port-Hamiltonian approach presents new techniques exploiting the energy balance or dissipativity of the underlying system, i.e,
\begin{displaymath}
    \frac{d}{dt}\cH(x)=0\quad\text{ or }\quad\frac{d}{dt}\cH(x)\leq 0,
\end{displaymath}
which are due to the algebraic constraints posed on the system. This is part of the reason for the recent popularity of the port-Hamiltonian approach and has led to many works in this area. For an overview of port-Hamiltonian systems, see the survey article \cite{survey_vanderschaft}. For a thorough introduction on deterministic equations of port-Hamiltonian structure, see the monographs \cite{book_jacob} and \cite{book_vanderschaft}.

On the other hand, stochastic port-Hamiltonian systems present a rather new field of research with notable works being done in the last decade \cite{fujimoto_satoh, fang_gao}. The methods are not fully explored yet and pose a wide range of unsolved problems, especially in the infinite-dimensional setting: First works in the direction of port-Hamiltonian SPDEs include the recent articles \cite{lw_1,lw_2}.

To introduce port-Hamiltonian SPDEs, let $H$ be a separable Hilbert space with orthogonal decomposition $H = H_q \oplus H_p$. Let $P_p\colon H\to H$ be the orthogonal projection to $H_p$, i.e., $P_p(q,p):= (0,p)$ for $q \in H_q\,, p \in H_p$.

Let $E$ be a Polish space, and $\gamma\colon H\times E \to H$ be measurable. Furthermore, let $F,\sigma\colon H\to H$ be Borel measurable. Let $(\Omega,\cF,(\cF_t)_t,\fP)$ be a filtered probability space rich enough so that all of the following notions are defined. The equation
\begin{equation}\label{FR_PHSPDE} 
\begin{split}
    dX^x_t&=(A\partial_x\cH(X^x_t) +P_pF(X^x_t) ) dt +  P_p\sigma(X^x_t)dW_t + \int_E  P_p\gamma(X^x_t,\eta)\widetilde{N}(dt,d\eta)
    \\
    X^x_0 &=x \in L^2(\Omega,\mathcal{F}_0,\mathbb{P};H)
\end{split}
\end{equation}
is a port-Hamiltonian SPDE with $W_\cdot$ a $Q$-Wiener process and $\widetilde{N}(dt,d\nu)$ an independent compensated Poisson random measure with compensator $\nu(dt,d\eta)=dt\mu(d\eta)$. Here $Q$ is a non-negative, symmetric, trace-class operator on a sub-Hilbert space $U \subset H$. Let us denote with $L_2^0:=L_2^0(H)$ the Hilbert space of Hilbert-Schmidt Operators from the separable Hilbert space $U_0:=Q^{1/2}U$ to $H$. Furthermore, $A:=J-R$, where
\begin{equation}\label{FR_equ:operatorA}
    R=\begin{pmatrix}
       R_q & 0
       \\
       0 & R_p
       \end{pmatrix}\quad\quad\text{and}\quad\quad J=\begin{pmatrix}
              0 & J_0\\
              -J_0^\ast & 0
             \end{pmatrix},
\end{equation}
\noindent $R_q,R_p$ are positive self-adjoint operators on $H_q$, $H_p$, respectively, $J_0\colon\operatorname{dom}(J_0)\subseteq H_p\to H_q$ such that  $-R$ and $J$ and the closure of $J-R$  generate quasi-contraction semigroups. As in the finite-dimensional case, the restriction of $R$ and $J$ comes from the considered model, as $R$ represents linear resistance and $J$ corresponds to the Hamiltonian structure.
For such systems, it is often assumed that
\begin{displaymath}
    \mathcal{H}(x):=\langle x,\tilde{H}x\rangle,
\end{displaymath}
with $\tilde{H}$ a bounded, positive self-adjoint operator. In this article, we consider the case where $\tilde{H}=\Id$. The energy for the total random system is $\mathbb{E}[\mathcal{H}(X_t)]$.

In \cite{farkas}, a class of SPDEs with multiple invariant measures on a separable Hilbert space $H$ was examined. Here, the case of an invariant subspace $H_0\subset H$ was considered, while the dynamics decayed on the orthogonal complement $H_1=H_0^\perp$. More precisely, it was assumed that there exists of a decomposition of the Hilbert space $H=H_0\oplus H_1$ such that the semigroup generated by the linear part of the SPDE leaves the space $H_0$ invariant, while the dynamics decay on $H_1$. Besides \cite{farkas}, there are other articles which consider SPDEs with similar invariance assumptions, for instance, \cite{vangaans}.

While (port-)Hamiltonian systems describing particle systems have a natural decomposition $H_0=H_q$ and $H_1=H_p$ into position and velocity coordinates, the invariance assumption is not suitable, as the Hamiltonian structure of the equations provides a natural interaction between the position and velocity variables via a skew-symmetric part of the drift. Therefore, we still assume the orthogonal decomposition but drop the assumption of invariance and allow interaction between $H_0$ and $H_1$ through the linear dynamics. The goal of this article is to find conditions under which stability properties of the system can be shown. Namely, we examine the existence and uniqueness of an invariant measure of the SPDE below. This will lead the way for future works to also analyse the case of multiple invariant measures by weakening the assumptions used for uniqueness.

In this article, we consider the equation
\begin{equation}\label{eq:spde_intro}
\begin{split}
    dX_t&=(AX_t+\hat{F}(X_t))dt+\hat{\sigma}(X_t)dW_t+\int_E\hat{\gamma}(X_t,\eta)\tilde{N}(dt,d\eta)
    \\
    X_0&\in L^2(\Omega,\cF,\fP;H)
\end{split}
\end{equation}
where $H=H_0\oplus H_1$ is a Hilbert space with orthogonal decomposition $H_0\perp H_1$ and the operator $A$ has the form
\begin{equation}\label{eq:operatorA_intro}
    A=\begin{pmatrix}
        -R_0&D_0
        \\
        D_1&-R_1
    \end{pmatrix}=:D-R.
\end{equation}
The processes $W$ and $N$ are given as above. As application, we keep in mind the classical system consisting of position and momentum, where the decomposition is denoted by $H_0=H_q$ and $H_1=H_p$.

For the case where $\hat{F}=P_pF$, $\hat{\sigma}=P_p\sigma$ and $\hat{\gamma}=P_p\gamma$ as well as the decomposition of $A$ given by \eqref{FR_equ:operatorA} with the corresponding assumptions on $R$ and $J$, equation \eqref{eq:spde_intro} coincides with \eqref{FR_PHSPDE}. Therefore, the considerations in this work cover the case of a port-Hamiltonian SPDE with Hamiltonian given by $\cH(x)=\frac{1}{2}\langle x,x\rangle$. We will explain this connection in more detail in Section \ref{subs:setting}.

The goals of this article are two-fold. First of all, we argue that the above system is in fact a special case of a so-called stochastic port-Hamiltonian system, motivating a further generalization in this direction in the future. Furthermore, we investigate the stability properties of Equation \eqref{eq:spde_intro}. Namely, we show the following result:
\begin{Theorem}\label{thm:main_intro}
Assume that the operator $A$ is as given in \eqref{eq:operatorA_intro}, where $D$ and $R$ fulfill the following estimate: There exist $\lambda_0,\lambda_1,\beta\in\R$ such that for all $x\in\cD(A)$, written as $x=(x_0,x_1)\in H_0\oplus H_1$,
\begin{align}
    -\langle R_ix_i,x_i\rangle_{H_i}&\leq-\lambda_i\|x_i\|_{H_i}^2\label{eq:condition_R}
    \\
    \langle Dx,x\rangle_H&\leq\beta\|x\|_H^2\label{eq:condition_D}
\end{align}
Also, assume that the functions $\hat{F},\hat{\sigma},\hat{\gamma}$ fulfill the following Lipschitz conditions for some constants $L_F,L_\sigma,L_\gamma$:
\begin{equation}\label{eq:lipschitz_intro}
\begin{split}
    \|\hat{F}(x)-\hat{F}(y)\|_H^2&\leq L_F\|x-y\|_H^2
    \\
    \|\hat{\sigma}(x)-\hat{\sigma}(y)\|_{L_2^0(H)}&\leq L_\sigma\|x-y\|_H^2
    \\
    \int_E\|\hat{\gamma}(x,\eta)-\hat{\gamma}(y,\eta)\|_H^2\mu(d\eta)&\leq L_\gamma\|x-y\|_H^2
\end{split}
\end{equation}
where $\|\cdot\|_{L_2^0(H)}$ denotes the Hilbert-Schmidt norm, and $\gamma$ satisfies the following moment conditions:
\begin{displaymath}
    \int_E\|\hat{\gamma}(0,\eta)\|_H^2\mu(d\eta)<\infty\text{ and }\int_E\|\hat{\gamma}(x,\eta)\|_H^4\mu(d\eta)<\infty\text{ for all }x\in H.
\end{displaymath}
Finally, assume that
\begin{displaymath}
    \varepsilon:=2\big[\min(\lambda_0,\lambda_1)-\beta-\sqrt{L_F}\big]-L_\sigma-L_\gamma>0.
\end{displaymath}
Then the mild solution $(X_t)_{t\geq 0}$ of \eqref{eq:spde_intro} possesses a unique invariant measure $\pi$. Furthermore, $(X_t)_{t\geq 0}$ converges exponentially fast to $\pi$ in Wasserstein 2-distance, i.e.,
\begin{displaymath}
    W_2(\rho_t,\pi)\leq e^{-\frac{\varepsilon}{2}t}W_2(\rho_0,\pi)
\end{displaymath}
where $(\rho_t)_{t\geq 0}$ denotes the law of $(X_t)_{t\geq 0}$ with initial distribution $\rho_0$.
\end{Theorem}
We are typically interested in the case $\lambda_0,\lambda_1,\beta>0$, i.e., the diagonal terms induce dissipation in form of linear resistance (e.g. friction), while the off-diagonal terms transfer energy between the coordinates.

The article is structured as follows. In Section 2, we recall the concept of a weak solution and place Equation \eqref{eq:spde_intro} in the port-Hamiltonian setting. Finally, we discuss the problem of invariant measures and review the main existence and uniqueness theorem for invariant measures. Section 3 is devoted to the proof of Theorem \ref{thm:main_intro} (Thm. \ref{thm:main_outro}). Additionally, we give comments on typical assumptions on the off-diagonal part $D$ which imply Estimate \eqref{eq:condition_D} (Remark \ref{rk:cond_D}).

\section{Preliminaries}

\subsection{Notation}

Let $H$ be a separable Hilbert space which admits an orthogonal decomposition $H=H_0\oplus H_1$. We use the following notational convention to denote operators on this decomposed space: For any $x\in H$, we write $(x_0,x_1)\in H_0\oplus H_1$, i.e., $x_0\in H_0$ and $x_1\in H_1$. An operator $A\colon\cD(A)\subset H\to H$ can be written in the following way:
\begin{displaymath}
    A=\begin{pmatrix}
    A_{00}&A_{01}\\
    A_{10}&A_{11}
    \end{pmatrix}
\end{displaymath}
where $A_{ij}\colon\cD(A_{ij})\subset H_j\to H_i$. The action of the operator $A$ works similarly to the two-dimensional matrix product, i.e. for an element $x=(x_0,x_1)\in\cD(A)$, we get
\begin{displaymath}
    Ax=\begin{pmatrix}
    A_{00}x_0+A_{01}x_1\\
    A_{10}x_0+A_{11}x_1.
    \end{pmatrix}
\end{displaymath}
We denote by $P_0,P_1\colon H\to H$ the projection operators on the respective subspaces, i.e., for any $x=(x_0,x_1)\in H$,
\begin{displaymath}
    P_0x=(x_0,0)\text{ and }P_1x=(0,x_1).
\end{displaymath}
By abuse of notation, we sometimes view $P_ix$ as an element of $H_i$, $i=0,1$ when it is convenient. Due to the orthogonality $H_0\perp H_1$, it is obvious that
\begin{displaymath}
    \langle P_0x,P_1y\rangle=\langle P_1x,P_0y\rangle=0
\end{displaymath}
for any $x,y\in H$, which also implies that
\begin{displaymath}
    \|x\|_H^2=\|P_0x\|_H^2+\|P_1x\|_H^2.
\end{displaymath}

\begin{Remark}
 By above considerations, it does not matter if we consider the $H_i$- or $H$-norm of an element from $H_i$ ($i=0,1$). Therefore, we will identify these notions whenever it makes sense, and omit the space specifying the norm.
\end{Remark}

\subsection{Setting}\label{subs:setting}

Let $H$ be a separable Hilbert space with orthogonal decomposition $H=H_0\oplus H_1$. Furthermore, let $E$ be a Polish space with $\sigma$-algebra $\cE$. We assume that the functions
\begin{displaymath}
    \hat{\gamma}\colon H\times E\to H,\ \hat{\sigma}\colon H\to H,\text{ and }\hat{F}\colon H \to H
\end{displaymath}
are measurable. Let $(\Omega,\cF,(\cF_t)_t,\fP)$ be a filtered probability space s.t. all notions below are defined. Consider the following SPDE:
\begin{equation}\label{eq:spde}
    \begin{split}
        dX_t&=(AX_t+\hat{F}(X_t))dt+\hat{\sigma}(X_t)dW_t+\int_E\hat{\gamma}(X_t,\eta)\tilde{N}(dt,d\eta)
        \\
        X_0&\in L^2(\Omega,\cF,\fP;H)
    \end{split}
\end{equation}
where $(W_t)_{t\geq 0}$ is a $Q$-Wiener process, $N$ is a Poisson random measure with compensator $\nu(dt,d\eta)=dt\mu(d\eta)$ for some $\sigma$-finite measure $\mu$, and $\tilde{N}(dt,d\eta)=N(dt,d\eta)-\nu(dt,d\eta)$. We assume that the processes $W$ and $N$ are independent. Furthermore, we assume that $(A,\cD(A))$ is of the form \eqref{eq:operatorA_intro} and generates a strongly continuous semigroup.

Under the assumptions of Theorem \ref{thm:main_intro}, Equation \eqref{eq:spde} posesses a mild solution in the following sense: Let $(S(t))_{t\geq 0}$ be the semigroup generated by the operator $A$. Then the process $(X_t)_t$ is called mild solution to \eqref{eq:spde} if the process satisfies $\fP$-a.s. the following equation:
\begin{align*}
    X_t&=S(t)X_0+\int_0^tS(t-s)\hat{F}(X_s)ds+\int_0^tS(t-s)\hat{\sigma}(X_s)dW_s
    \\
    &\hspace{20pt}+\int_0^t\int_ES(t-s)\hat{\gamma}(X_s,\eta)\tilde{N}(ds,d\eta),\ t\geq 0.
\end{align*}
For a more detailed description of the solution theory, see \cite{amr09,ftt10,book_mr14}.

Next, we place our equation in the setting of the port-Hamiltonian theory. Namely, if we assume that the Hamiltonian is given by $\cH(x)=\frac{1}{2}\langle x,x\rangle$, then the port-Hamiltonian SPDE can be viewed as a special case of Equation \eqref{eq:spde}.

Let $A$ be as in \eqref{eq:operatorA_intro} with $D_1=-D_0^\ast$ and $R$ fulfilling condition \eqref{eq:condition_R} with $\lambda_i\geq 0$. Furthermore, assume that
\begin{displaymath}
    \hat{\gamma}\colon H\times E\to \{0\}\oplus H_1,\ \hat{\sigma}\colon H\to \{0\}\oplus H_1,\text{ and }\hat{F}\colon H \to \{0\}\oplus H_1.
\end{displaymath}
Since $\hat{F},\hat{\sigma}$ and $\hat{\gamma}$ only act on the space $H_1$, we may identify them with the projections $P_1F,P_1\sigma$ and $P_1\gamma$. Then the following equation is a port-Hamiltonian SPDE:
\begin{equation}\label{eq:ph_spde}
    \begin{split}
        dX_t&=(A\partial_x\cH(X_t)+P_1F(X_t))dt+P_1\sigma(X_t)dW_t+\int_EP_1\gamma(X_t,\eta)\tilde{N}(dt,d\eta)
        \\
        X_0&\in L^2(\Omega,\cF,\fP;H)
    \end{split}
\end{equation}
where $\partial_x$ denotes the Fr\'echet derivative. Then $R$ is positive and $D=J$ is skew-symmetric, and hence, the linear part given by $A$ has the shape needed for \eqref{eq:operatorA_phs}. The nonlinear drift as well as the noise terms act as ports and only influence a part of the system, as they are projected on the space $H_1$. Note that it is usual to take the energy being $\cH(x)=\frac{1}{2}\langle x,\tilde{H} x\rangle$ for some bounded elliptic operator $\tilde{H}$. The Fr\'echet derivative then becomes $\tilde{H} x$, by identifying the derivative with the corresponding element of the Hilbert space.

Finally, in addition to prescribing the shape of $A$, let $\tilde{H}=\Id$. The energy then becomes $\cH(x)=\frac{1}{2}\|x\|^2$, and $\partial_x\cH(x)=x$. Since we consider a classical particle system consisting of position and momentum, we take $H_0=H_1$. By taking this special case, we see that the port-Hamiltonian system \eqref{eq:ph_spde} becomes the SPDE \eqref{eq:spde}, which is the object of our study.

\subsection{Invariant Measures}

The study of invariant measures plays an important role in the analysis of the long-time behaviour of the process. The existence of an invariant measure is the first step in the analysis of the limiting behaviour, as the invariant measure is the candidate for the limiting distribution of the process.

Assume that we are given a Markov process $(X_t^x)_{t\geq 0}$ with transition probabilities $p_t(x,dy)=\fP(X_t^x\in dy)$. As application, we have in mind the mild solution of Equation \eqref{eq:spde}. We also denote by $(p_t)_t$ the transition semigroup of the process, i.e., for any bounded measurable function $f\colon H\to\R$, set
\begin{equation}\label{eq:semigroup}
    p_tf(x):=\int_H f(y)p_t(x,dy),\ t\geq 0,x\in H.
\end{equation}
Under the conditions of Theorem \ref{thm:main_intro}, for any $X_0\in L^2(\Omega,\cF_0,\fP;H)$, there exists a mild solution to equation \eqref{eq:spde}. Furthermore, for $X_0,Y_0\in L^2(\Omega,\cF_0,\fP;H)$ with corresponding mild solutions $(X_t)_{t\geq 0}$ and $(Y_t)_{t\geq 0}$, we have 
\begin{displaymath}
    \E[\|X_t-Y_t\|^2]\leq C(T)\E[\|X_0-Y_0\|^2].
\end{displaymath}
This implies that the semigroup defined by \eqref{eq:semigroup} is $C_b$-Feller. For a more detailed view on integration w.r.t. Levy processes as well as solution theory, see e.g. \cite{amr09, book_mr14}.

For $t\geq 0$, denote by $p_t^\ast$ the adjoint to the operator $p_t$, acting on the space of probability measures as
\begin{displaymath}
    p_t^\ast\rho(dx)=\int_Hp_t(y,dx)\rho(dy).
\end{displaymath}
We call a probability measure $\pi$ on $(H,\cB(H))$ invariant measure for the process $(X_t)_t$ (or for the semigroup $(p_t)_t$) if for all $t\geq 0$, we have $p_t^\ast\pi=\pi$.

The goal of this work is to investigate under which conditions there exists a unique invariant measure for the mild solution $(X_t^x)_t$ to \eqref{eq:ph_spde}. To this end, we consider the Wasserstein-2-space $\cP_2(H)$ of all probability measures on $(H,\cB(H))$ with finite second moments, equipped with the Wasserstein-2-distance
\begin{displaymath}
    W_2(\rho,\tilde{\rho})=\Big(\inf_{G\in\Pi(\rho,\tilde{\rho})}\int_{H\times H}\|x-y\|^2G(dx,dy)\Big)^{\frac{1}{2}}.
\end{displaymath}
where $\Pi(\rho,\tilde{\rho})$ denotes the set of all couplings between $\rho$ and $\tilde{\rho}$. It is well-known that $(\cP_2(H),W_2)$ it itself a Polish space if $H$ is Polish. For a more detailed introduction to Wasserstein distances, see e.g. \cite[Section 6]{villani_oldnew}.

For the existence of a unique invariant measure for the solution, we have the following result. It can be found for various cases in \cite{daprato_zabczyk,book_mr14,peszat_zabczyk,mr23}. For convenience, we sketch the proof here.

\begin{Theorem}\label{thm:uim}
Assume the following:
\begin{enumerate}
    \item $\hat{F}$ and $\hat{\sigma}$ are Borel measurable, $\hat{\gamma}$ is measurable from $\cB(H)\otimes\cE$ to $\cB(H)$.
    \item There exist constants $L_F,L_\sigma,L_\gamma>0$ such that
    \begin{equation}\label{eq:lipschitz_invariant}
        \begin{split}
        \|\hat{F}(x)-\hat{F}(y)\|_H^2&\leq L_F\|x-y\|_H^2
        \\
        \|\hat{\sigma}(x)-\hat{\sigma}(y)\|_{L_2^0(H)}^2&\leq L_\sigma\|x-y\|_H^2
        \\
        \int_E\|\hat{\gamma}(x,\eta)-\hat{\gamma}(y,\eta)\|_H^2\mu(d\eta)&\leq L_\gamma\|x-y\|_H^2
        \end{split}
    \end{equation}
    \item $\displaystyle\int\|\hat{\gamma}(0,\eta)\|_H^2\mu(d\eta)<\infty$
    \item $A$ generates a pseudo-contraction semigroup, i.e. there exists $\omega\in\R$ such that $\|S_A(t)\|\leq e^{\omega t}$ for all $t\geq 0$.
    \item We assume that the drift fulfills the following dissipativity condition: There exists $\alpha>0$ such that
    \begin{displaymath}
        \langle A(x-y),x-y\rangle+\langle \hat{F}(x)-\hat{F}(y),x-y\rangle\leq-\alpha\|x-y\|^2\text{ for all }x,y\in\cD(A)
    \end{displaymath}
    \item $\varepsilon:=2\alpha-L_\gamma-L_\sigma>0$.
    \item For all $x\in H$, we have
    \begin{displaymath}
        \int_E\|\hat{\gamma}(x,\eta)\|_H^4\mu(d\eta)<\infty.
    \end{displaymath}
\end{enumerate}
Then for any $\rho,\tilde{\rho}\in\cP_2(H)$, we have
\begin{displaymath}
    W_2(p_t^\ast\rho,p_t^\ast\tilde{\rho})\leq W_2(\rho,\tilde{\rho})e^{-\varepsilon t/2}
\end{displaymath}
In particular, the Markov process determined by \eqref{eq:spde} has a unique invariant measure $\pi$. This measure has finite second moments and it holds that
\begin{displaymath}
    W_2(p_t^\ast\rho,\pi)\leq W_2(\rho,\pi)e^{-\varepsilon t/2}
\end{displaymath}
for each $\rho\in\cP_2(H)$.
\end{Theorem}
For the proof, recall the following properties of the Wasserstein distance:
\begin{Lemma}\label{wasserstein_lemma}
\begin{enumerate}
    \item The infimum is attained, i.e. there exists $\hat{G}\in\Pi(\rho,\tilde{\rho})$ such that
    \begin{displaymath}
     W_p(\rho,\tilde{\rho})=\Big(\int_{H\times H}\|x-y\|^p \hat{G}(dx,dy)\Big)^{1/p}
    \end{displaymath}
    \item The following estimate holds: For all kernels $p$ and all couplings $G\in\Pi(\rho,\tilde{\rho})$,
    \begin{displaymath}
     W_p\Big(\int_H p(y,\cdot)\rho(dy),\int_H p(y,\cdot)\tilde{\rho}(dy)\Big)\leq\int_{H\times H}W_p(p(x,\cdot),p(y,\cdot))G(dx,dy)
    \end{displaymath}
\end{enumerate}
\end{Lemma}
\begin{proof}
Let $X_t^x,X_t^y$ be the mild solutions to \eqref{eq:spde} with initial value $x$ and $y$, respectively. By setting $P_1\equiv 0$ in \cite[Prop. 3.5]{farkas}, the above conditions imply that
\begin{equation}\label{eq:mean_square}
    \E[\|X_t^x-X_t^y\|^2]\leq e^{-\varepsilon t}\|x-y\|_H^2
\end{equation}
for all $x,y\in H$, see also \cite{book_mr14, mr23}. Property \eqref{eq:mean_square} is also called exponential mean-square stability. Using this property, we show that the corresponding evolution of probability measures is Cauchy with respect to the Wasserstein-2-distance.
Fix $t$ and let $G_t=\cL(X_t^x,X_t^y)$ be the joint distribution of $(X_t^x,X_t^y)$. Then $G_t$ is a coupling of $X_t^x$ and $X_t^y$. By the definition of the Wasserstein distance and the stability estimate \eqref{eq:mean_square}, we have
\begin{align*}
 W_2(p_t(x,\cdot),p_t(y,\cdot))&=\inf_{G\in\Pi(p_t(x,\cdot),p_t(y,\cdot))}\Big(\int_{H\times H}\|x-y\|^2G(dx,dy)\Big)^{1/2}
 \\
 &\leq\Big(\int_{H\times H}\|x-y\|^2G_t(dx,dy)\Big)^{1/2}
 \\
 &=\big(\E[\|X_t^x-X_t^y\|^2]\big)^{1/2}
 \\
 &\leq e^{-\frac{\varepsilon}{2}t}\|x-y\|,
\end{align*}
Using Lemma \ref{wasserstein_lemma}, we obtain for two arbitrary probability measures $\rho,\tilde{\rho}$
\begin{equation}\label{proof_estimate_1}
\begin{split}
    W_2(p_t^\ast\rho,p_t^\ast\tilde{\rho})&\leq\int_{H\times H}W_2(p_t(x,\cdot),p_t(y,\cdot))\hat{G}(dx,dy)
    \\
    &\leq e^{-\frac{\varepsilon}{2}t}\int_{H\times H}\|x-y\|\hat{G}(dx,dy)
    \\
    &\leq e^{-\frac{\varepsilon}{2}t}\Big(\int_{H\times H}\|x-y\|^2\hat{G}(dx,dy)\Big)^{1/2}
    \\
    &=e^{-\frac{\varepsilon}{2}t}W_2(\rho,\tilde{\rho})
\end{split}
\end{equation}
where $\hat{G}$ is the coupling given by Lemma \ref{wasserstein_lemma}.

Next, we show that the sequence $(p_k^\ast\rho)_{k\geq 1}$ is Cauchy for any arbitrary $\rho\in\cP_2(H)$. Let $l>k$. Then
\begin{align*}
    W_2(p_k^\ast\rho,p_l^\ast\rho)&\leq\sum_{j=k}^{l-1}W_2(p_{j+1}^\ast\rho,p_j^\ast\rho)\stackrel{\eqref{proof_estimate_1}}\leq\sum_{j=k}^{l-1}e^{-\frac{\varepsilon}{2}j}W_2(p_1^\ast\rho,\rho)\xrightarrow{k\to\infty}0.
\end{align*}
Since $\cP_2(H)$ is complete w.r.t. $W_2$, there exists a limit $\pi\in\cP_2(H)$. Next, we show that $\pi$ is invariant, i.e., $p_t^\ast\pi=\pi$ for all $t>0$. But this holds since
\begin{align*}
    W_2(p_t^\ast\pi,\pi)&\leq W_2(p_t^\ast\pi,p_t^\ast p_k^\ast\rho)+W_2(p_t^\ast p_k^\ast\rho,p_k^\ast\rho)+W_2(p_k^\ast\rho,\pi)
    \\
    &\leq e^{-\frac{\varepsilon}{2}t}\underbrace{W_2(\pi,p_k^\ast\rho)}_{\to 0, k\to\infty}+\underbrace{e^{\frac{\varepsilon}{2}k}}_{\to 0, k\to\infty}W_2(p_t^\ast\rho,\rho)+\underbrace{W_2(p_k^\ast\rho,\pi)}_{\to 0,k\to\infty}
    \\
    &\to 0, k\to\infty.
\end{align*}
Hence, $\pi$ is an invariant measure. It is left to check that $\pi$ is unique. Let $\tilde{\pi}$ be another invariant measure. We have
\begin{displaymath}
 W_2(\pi,\tilde{\pi})=W_2(p_t^\ast\pi,p_t^\ast\tilde{\pi})\leq e^{-\frac{\varepsilon}{2}t}W_2(\pi,\tilde{\pi})
\end{displaymath}
Since $t>0$ is arbitrary, uniqueness follows.
\end{proof}

\section{Proof of Theorem \ref{thm:main_intro}}

This section is now devoted to the proof of the main result. As a reminder, let $(\Omega,\cF, (\cF_t)_t,\fP)$ be a filtered probability space s.t. all expressions appearing are defined. We consider the SPDE
\begin{displaymath}
 dX_t=(AX_t+\hat{F}(X_t))dt+\hat{\sigma}(X_t)dW_t+\int_E\hat{\gamma}(X_t,\eta)\tilde{N}(dt,d\eta)
\end{displaymath}
where $(W_t)_t$ is a $Q$-Wiener process, $N$ is a Poisson process with compensator $\nu(dt,d\eta)=dt\mu(d\eta)$, the functions $\hat{F}\colon H\to H,\hat{\sigma}\colon H\to H,\hat{\gamma}\colon H\times E\to H$ are measurable on their respective spaces and fulfill the Lipschitz conditions \eqref{eq:lipschitz_intro}, and the operator $A$ has the following form:
\begin{equation}\label{eq:operatorA}
    A=\begin{pmatrix}
        -R_0&D_0
        \\
        D_1&-R_1
    \end{pmatrix}=:D-R
\end{equation}
We assume the following bound:
\begin{Assumption}\label{assumption_R}
 For $i=0,1$, $R_i\colon\cD(R_i)\subset H_i\to H_i$ and there exists $\lambda_i\in\R$ such that
 \begin{displaymath}
     \langle R_ix_i,x_i\rangle_{H_i}\geq\lambda_i\|x_i\|_{H_i}^2.
 \end{displaymath}
\end{Assumption}

\begin{Proposition}\label{prop:dc}
Consider an operator of the form \eqref{eq:operatorA} which fulfills Assumption \ref{assumption_R}. Assume that the Lipschitz condition \eqref{eq:lipschitz_intro} holds. Furthermore, assume that there exists $\beta\in\R$ such that for all $x\in \cD(A)$,
\begin{equation}\label{eq:condition_beta}
    \langle Dx,x\rangle\leq\beta\|x\|^2.
\end{equation}
Set $\alpha:=\min(\lambda_0,\lambda_1)-\beta$. Then
\begin{displaymath}
    \langle A(x-y),x-y\rangle\leq-\alpha\|x-y\|^2.
\end{displaymath}
Additionally, if we assume that
\begin{equation}\label{eq:condition_coefficients}
 a:=\alpha-\sqrt{L_F}=\min(\lambda_0,\lambda_1)-\beta-\sqrt{L_F}>0,
\end{equation}
then the drift is dissipative: For all $x,y\in \cD(A)$,
\begin{equation}\label{eq:gdc}
    \langle Ax-Ay,x-y\rangle+\langle \hat{F}(x)-\hat{F}(y),x-y\rangle\leq -a\|x-y\|_H^2.
\end{equation}
\end{Proposition}
\begin{Remark}\label{rk:cond_D}
The condition on $D$ is especially fulfilled in the following two cases:
\begin{enumerate}
    \item $D_0,D_1$ bounded: Assuming that the operators are bounded, we obtain
    \begin{align*}
        \langle D(x-y),x-y\rangle&=\langle D_0(P_1x-P_1y),P_0x-P_0y\rangle+\langle D_1(P_0x-P_0y),P_1x-P_1y\rangle
        \\
        &\leq\|D_0(P_1x-P_1y)\|\|P_0x-P_0y\|+\|D_1(P_0x-P_0y)\|\|P_1x-P_1y\|
        \\
        &\leq(\|D_0\|+\|D_1\|)\|P_0x-P_0y\|\|P_1x-P_1y\|
        \\
        &\leq\frac{\|D_0\|+\|D_1\|}{2}(\|P_0x-P_0y\|^2+\|P_1x-P_1y\|^2)
    \end{align*}
    for all $x,y\in H$, where we used Young's inequality in the last step. Hence, Condition \eqref{eq:condition_beta} holds with
    \begin{displaymath}
        \beta=\frac{\|D_0\|+\|D_1\|}{2}.
    \end{displaymath}
    \item Case $D_0^\ast=-D_1$: In this case, the matrix $D$ is skew-symmetric and the off-diagonal terms in the proof below vanish, i.e., condition \eqref{eq:condition_beta} holds with $\beta=0$. This is exactly the case in a classical Hamiltonian or port-Hamiltonian system consisting of position and velocity coordinates.
\end{enumerate}
\end{Remark}
\begin{proof}[Proof of Proposition \ref{prop:dc}]
We first estimate the $A$-term, then use the Lipschitz continuity of $F$ for the remainder. By a slight abuse of notation, we write
\begin{align*}
    \Big\langle&\begin{pmatrix}-R_0&D_0\\ D_1&-R_1\end{pmatrix}\begin{pmatrix}P_0x-P_0y\\ P_1x-P_1y\end{pmatrix},\begin{pmatrix}P_0x-P_0y\\ P_1x-P_1y\end{pmatrix}\Big\rangle
    \\
    &=\langle-R_0(P_0x-P_0y),P_0x-P_0y\rangle+\langle-R_1(P_1x-P_1y),P_1x-P_1y\rangle
    \\
    &\hspace{20pt}+\langle D_0(P_1x-P_1y),P_0x-P_0y\rangle+\langle D_1(P_0x-P_0y),P_1x-P_1y\rangle
\end{align*}
Using Assumption \ref{assumption_R} and Condition \eqref{eq:condition_beta}, we obtain
\begin{align*}
    \langle A(x-y),x-y\rangle&=\langle-R(x-y),x-y\rangle+\langle D(x-y),x-y\rangle
    \\
    &=\langle -R_0(P_0x-P_0y),P_0x-P_0y\rangle+\langle -R_1(P_1x-P_1y),P_1x-P_1y\rangle
    \\
    &\hspace{20pt}+\langle D(x-y),x-y\rangle
    \\
    &\leq-\lambda_0\|P_0x-P_0y\|^2-\lambda_1\|P_1x-P_1y\|^2+\beta\|x-y\|^2
    \\
    &\leq-\min(\lambda_0,\lambda_1)(\|P_0x-P_0y\|^2+\|P_1x-P_1y\|^2)+\beta\|x-y\|^2
    \\
    &=-\big(\min(\lambda_0,\lambda_1)-\beta\big)\|x-y\|^2
\end{align*}
Combining the above calculations with the Lipschitz bound on $\hat{F}$, we obtain
\begin{align*}
    \langle& A(x-y),x-y\rangle+\langle\hat{F}(x)-\hat{F}(y),x-y\rangle
    \\
    &\leq-\alpha\|x-y\|^2+\sqrt{L_F}\|x-y\|^2
    \\
    &=-a\|x-y\|^2.
\end{align*}
\end{proof}

\begin{Remark}
 Estimate \eqref{eq:gdc} holds for general coefficients, but to have a useful dissipation, we need Assumption \eqref{eq:condition_coefficients}.
\end{Remark}
Collecting the above statements, we are now ready to state our main result.

\begin{Theorem}\label{thm:main_outro}
Additionally to the assumptions of Proposition \ref{prop:dc}, assume the Lipschitz conditions \eqref{eq:lipschitz_intro} and
\begin{displaymath}
    \varepsilon:=2a-L_\sigma-L_\gamma>0
\end{displaymath}
as well as the moment assumptions
\begin{displaymath}
    \int_E\|\hat{\gamma}(0,\eta)\|^2\mu(d\eta)<\infty\text{ and }\int_E\|\hat{\gamma}(x,\eta)\|^4\mu(d\eta)<\infty\text{ for all }x\in H.
\end{displaymath}
Then the mild solution $(X_t)_{t\geq 0}$ of \eqref{eq:spde} possesses a unique invariant measure $\pi$. Furthermore, $(X_t)_{t\geq 0}$ converges exponentially fast to $\pi$ in Wasserstein 2-distance, i.e.,
\begin{displaymath}
    W_2(\rho_t,\pi)\leq e^{-\frac{\varepsilon}{2}t}W_2(\rho_0,\pi)
\end{displaymath}
where $(\rho_t)_{t\geq 0}$ denotes the law of $(X_t)_{t\geq 0}$ with initial distribution $\rho_0$.
\end{Theorem}
\begin{proof}
    The result is a consequence of the conclusion of Proposition \ref{prop:dc} together with Theorem \ref{thm:uim}.
\end{proof}

\noindent\textbf{Acknowledgements.} We thank B\'alint Farkas for fruitful discussions in the development of this article.


\begin{thebibliography}{1}

    \bibitem{amr09}
    S. Albeverio, V. Mandrekar, B. R\"udiger, Existence of mild solutions for stochastic differential equations and semilinear equations weith non-Gaussian L\'evy noise, Stoch. Process. Appl. 119(3), 835--863 (2009).

    \bibitem{daprato_zabczyk}
    G. Da Prato, J. Zabczyk: \emph{Stochastic Equations in Infinite Dimensions}. Encyclopedia of Mathematics and Its Applications. Cambridge Univ. Press, Cambridge (2014).
    
    \bibitem{FR_DOPT94}
    A. De Masi, E. Orlandi, E. Presutti, L. Triolo, Glauber evolution with the Kac potentials. I. Mesoscopic and macroscopic limits, interface dynamics, Nonlinearity 7(3), 633--696 (1994).
    
    \bibitem{FR_DOPT96}
    A. De Masi, E. Orlandi, E. Presutti, L. Triolo, Glauber evolution with Kac potentials. II. Fluctuations, Nonlinearity 9(1), 27--51 (1996).
    
    \bibitem{FR_bookDMSB}
    V. Duindam, S. Stramigioli: \emph{Modeling and Control of Complex Physical Systems: The Port-Hamiltonian Approach}. Springer Berlin, Heidelberg (2009).
    
    \bibitem{engel_nagel}
    K.J. Engel, R. Nagel: \emph{One-Parameter Semigroups for Linear Evolution Equations}. Springer New York (2000).
    
    \bibitem{fang_gao}
    Z. Fang, Ch. Gao, Stabilization of Input-Disturbed Stochastic Port-Hamiltonian Systems via Passivity, IEEE Trans. Automat. Control 62 (2017), no. 8, 4159--4166.

    \bibitem{farkas}
    B. Farkas, M. Friesen, B. R\"udiger, D. Schroers, On a class of SPDE with multiple invariant measures, NoDEA 2021.
    
    \bibitem{ftt10}
    D. Filipovi\'c, S. Tappe, J. Teichmann, Jump-diffusions in Hilbert spaces: existence, stability and numerics, Stochastics 82(5) (2010) 475--520.
    
    \bibitem{fujimoto_satoh}
    K. Fujimoto, S. Satoh, Passivity based control of stochastic port-Hamiltonian systems, IEEE Trans. Automat. Control 58 (2013), no. 5, 1139--1153.
    
    \bibitem{FR_GLPbook}
    G. Giacomin, J. L. Lebowitz, E. Presutti, Deterministic and stochastic hydrodynamic equations arising from simple microscopic model systems, Stochastic partial differential equations: six perspectives, Math. Surveys Monogr. 64 107--152Amer. Math. Soc., Providence, RI (1999).
    
    \bibitem{book_jacob}
    B. Jacob, H. Zwart: \emph{Linear Port-Hamiltonian Systems on Infinite-Dimensional Spaces}. Birkh\"auser (2012).
    
    \bibitem{stochastics_publication}
    P. Kuchling, B. R\"udiger, B. Ugurcan, Stability properties of some port-Hamiltonian SPDEs. Stochastics (2024), 1-–15. \url{https://doi.org/10.1080/17442508.2024.2387773}

    \bibitem{lw_1}
    F. Lamoline, J.J. Winkin, On stochastic port-Hamiltonian systems with boundary control and observation, 2017 IEE 56th Annual Conference on Decision and Control.
    
    \bibitem{lw_2}
    F. Lamoline, J.J. Winkin, Well-posedness of boundary controlled and observed stochastic port-Hamiltonian systems, IEEE Trans. Automat. Control 65 (2020), no. 10, 4258--4264.
    
    \bibitem{FR_LP63}
    J.L. Lebowitz, O.Penrose, Rigorous treatment of the van der Waals-Maxwell theory of the liquid-vapor transition, J. Mathematical Phys. 7, 98--113 (1966).
    
    \bibitem{book_mr14}
    V. Mandrekar, B. R\"udiger: \emph{Stochastic Integration in Banach Spaces: Theory and Applications} (Probability Theory and Stochastic Modelling book 73). Springer (2014)
    
    \bibitem{mr23}
    V. Mandrekar, B. R\"udiger, Stability properties of mild solutions of SPDEs related to pseudo differential equations, arXiv preprint, 2301.05120. To appear in: \emph{Quantum and Stochastic Mathematical Physics: Sergio Albeverio, Adventures of a Mathematician, Verona, Italy, March 25--29, 2019}, Springer, Berlin (2023).
    
    \bibitem{peszat_zabczyk}
    S. Peszat, J. Zabczyk: \emph{Stochastic Partial Differential Equations with L\'evy Noise}, Encyclopedia of Mathematics and its Applications, vol. 113, Cambridge University Press, Cambridge 2007.
    
    \bibitem{survey_vanderschaft}
    A. van der Schaft, Port-Hamiltonian systems: an introductory survey, Proceedings of the International Congress of Mathematicians, Madrid, Spain, 2006.
    
    \bibitem{book_vanderschaft}
    A. van der Schaft, D. Jeltsema: \emph{Port-Hamiltonian Systems Theory: An Introductory Overview}, now, 2014.
    
    \bibitem{vangaans}
    O. van Gaans, Invariant measures for stochastic evolution equations with Hilbert space valued L\'evy noise, preprint 2005, \url{https://www.math.leidenuniv.nl/~vangaans/gaansrep1.pdf}
    
    \bibitem{villani_oldnew}
    C. Villani: \emph{Optimal Transoprt: Old and New}. Springer, Berlin (2009).
    

\end{thebibliography}
\end{document}